\documentclass[16pt]{amsart}

\usepackage{amsfonts,amssymb,amscd,amsmath,enumerate,verbatim,calc}
\usepackage{color,soul}
\usepackage{mathtools}
\usepackage{graphicx}
\everymath{\displaystyle}
\usepackage{xcolor}
\usepackage[colorlinks,citecolor=  black]{hyperref}
\usepackage{ragged2e}
\usepackage{fancyhdr}
\usepackage{tikz}
\usetikzlibrary{shapes.geometric, arrows}
\tikzstyle{arrow} = [thick,->,>=stealth]
\tikzstyle{process} = [rectangle, minimum width=4cm, minimum height=2cm, text centered, text width=2.5cm, draw=green, fill=green!10]
\newcommand{\Hom}{\text{Hom}}
\newcommand{\ran}{\text{Im }}

\newcommand{\IFF}{\text{if and only if}}

\newcommand{\R}{\mathbb{R} }
\newcommand{\C}{\mathbb{C} }

\theoremstyle{plain}

\newtheorem{theorem}{Theorem}[section]
\newtheorem{corollary}[theorem]{Corollary}

\newtheorem{proposition}[theorem]{Proposition}

\theoremstyle{definition}
\newtheorem{definition}[theorem]{Definition}

\newtheorem{remark}[theorem]{Remark}
\newtheorem{example}[theorem]{Example}

\begin{document}
\title{A New Perspective on Eigenvalues and Eigenvectors of Bicomplex Linear Operators}
\author{Anjali}
\email{anjalisharma773@gmail.com}
\address{Department of Applied Mathematics, Gautam Buddha University, Greater Noida, Uttar Pradesh 201312, India}

\author{Akhil Prakash}
\email{akhil.sharma140@gmail.com}
\address{Department of Mathematics, Aligarh Muslim University, Aligarh, Uttar Pradesh 202002, India}

\author{Amita}
\email{amitasharma234@gmail.com}
\address{Department of Mathematics, Indira Gandhi National Tribal University, Amarkantak, Madhya Pradesh 484886, India}

\author{Prabhat Kumar}
\email{prabhatphilosopher@gmail.com}
\address{Department of Applied Mathematics, Gautam Buddha University, Greater Noida, Uttar Pradesh 201312, India}

\subjclass[IMS]{Primary 15A04, 15A30; Secondary 30G35}
\date{\today}
	
\begin{abstract}
This paper deals with eigenvalues and eigenvectors of bicomplex linear operators defined on bicomplex space. We investigate the properties of these operators in the context of eigenvalues and eigenvectors, along with some relevant theorems. Several theorems are explored to establish conditions for bicomplex eigenvalues and eigenvectors. Additionally, we examine the structure of the eigenspaces corresponding to these eigenvalues,analyze their properties, and discuss relevant theorems.
\end{abstract}
\maketitle
\section*{Introduction}
The field of bicomplex numbers has been a focus point of ongoing research in mathematics, experiencing significant advancements in recent years. Numerous researchers refer to \cite{alpay2023interpolation, alpay2014basics, wagh2024rank,Amita2018,luna2024singularities, luna2015, rochon2004}) have made substantial contributions to this domain. They have explored various avenues to analyze the properties of bicomplex numbers and develop concepts that adhere to a unified approach in multivariate complex number theory. Throughout this process, key ideas such as bicomplex topology  \cite{srivastava2008}, differentiability and analyticity of bicomplex functions \cite{price2018}, power series, bicomplex matrices, bicomplex Riemann zeta function and Dirichlet series, Cauchy's theory, and more, have been established with essential refinements. Furthermore, we have drawn upon existing literature, particularly the foundational concepts presented in the book  \cite{price2018} on bicomplex functions, incorporating relevant results and symbols.

Section 1 introduces the concept of bicomplex numbers and discusses fundamental definitions. Furthermore, we have discussed the bicomplex matrices and linear transformation. Also, we discussed the idempotent product in this section. 

Section 2 delves into eigenvalues and eigenvectors on bicomplex spaces. Here, we focus on the eigenvalues and eigenvectors of a particular type of  \;$\C_1$-linear maps  $T=e_1T_1 + e_2T_2$, and we investigate the related results and properties of $T$.

Section 3 discusses eigenspaces and modified eigenspaces of bicomplex  $\C_1$-linear maps $T=e_1T_1 +e_2T_2$. It also demonstrates the relationship between the set of eigenvalues and the set of modified eigenvalues of such type of bicomplex $\C_1 $-linear maps.

\section{Preliminaries and Notations} 
This section introduces bicomplex numbers, covering key concepts such as idempotent representation and the cartesian product of bicomplex space. It provides a comprehensive overview of these fundamental notions and presents essential results related to bicomplex numbers.\\

\noindent {\bf Bicomplex numbers:}
A bicomplex number is an element that can be expressed in the following form:
\[
\xi = {u}_{1}  + {i}_{1} {u}_{2}+ {i}_{2} {u}_{3} + {i}_{1} {i}_{2} {u}_{4},\; \;{u}_{k}\in \R, \;\;1\leqslant k\leqslant 4, \;\; \mbox{with}\;\; \;{i}_{1} {i}_{2} \ = \ {i}_{2} {i}_{1},\; \; {i}_{1}^2 = {i}_{2}^2  =  -1, 
\]
where $u_k$ are real numbers with $1\leqslant k\leqslant 4$, and the unit vectors $i_1$ and $i_2$ satisfy the conditions $i_1i_2 = i_2i_1$ and $i_1^2 = i_2^2 = -1$. The set of all bicomplex numbers is denoted as $\mathbb{C}_2$, referred to as the bicomplex space. The symbols $\mathbb{C}_{1},\; \mathbb{C}_{0}$, for convenience, denote the set of all complex numbers and the set of all real numbers respectively. The bicomplex space $\C_2$ can be described in the following two ways:\\
By considering the real components form:
\begin{eqnarray*}
\mathbb{C}_{2} &\coloneqq & \{{u}_{1}  + {i}_{1} {u}_{2}  + {i}_{2} {u}_{3}  + {i}_{1} {i}_{2} {u}_{4}\;:\; {u}_{1} , {u}_{2} , {u}_{3} , {u}_{4} \in\mathbb{C}_{0}\} 
\end{eqnarray*}
and the complex components form:
\begin{eqnarray}
\mathbb{C}_{2} &\coloneqq & \{{z}_{1} + {i}_{2} {z}_{2}\;:\;  {z}_{1}, {z}_{2} \in \mathbb{C}_{1}\}.
\end{eqnarray}
It is important to note that $\mathbb{C}_2$ contains zero-divisors, which means it is not a field but rather an algebra over $\mathbb{C}_1$. In $\mathbb{C}_2$, there are exactly four idempotent elements; $0, 1, e_1,$ and $e_2$. The nontrivial idempotent elements $e_1$ and $e_2$ are defined as follows:
 \[ e_{1} \coloneqq \frac{(1 + i_{1} i_{2})}{2}\quad  \mbox{and} \quad e_{2} \coloneqq \frac{(1 - i_{1} i_{2})}{2}.\]
Notice that $e_{1} + e_{2} =1$ , \;$e_{1}e_{2}  = e_{2}e_{1} =0$ and $e_{1}^2 =  e_{1},\; e_{2}^2 = e_{2}$. 
\begin{definition} \cite{price2018} (Principal Ideal) 
   The proper principal ideals generated by $e_1$ and $e_2$ are denoted as $I_1$ and $I_2$ and defined as:  
\begin{eqnarray*}
I_1=:\{\xi e_1:\xi \in \C_2\}=\{\xi^- e_1:\xi^- \in \C_1\},\\
I_2=:\{\xi e_2:\xi \in \C_2\}=\{\xi^+ e_2:\xi^+ \in \C_1\}. 
\end{eqnarray*}
\end{definition}
\noindent{\bf Singular and non-singular elements:} 
An element $\xi \in \C_2$ is said to be a non-singular if there exists $\eta \in \C_2$ such that $\xi. \eta = 1$. In other words, $\xi$ possesses a multiplicative inverse within $\C_2$. Also, an element in ${\C}_2$ that lacks such a multiplicative inverse is called a singular element. Moreover, there is an infinite number of elements that lack a multiplicative inverse.\\
Notably, $\xi=z_1 +i_2 z_2$ is classified as singular if and only if the absolute value of $z_1 +i_2 z_2$  equals zero, written as $|z_1^2 + z_2^2| = 0$. The set of all singular elements in $\C_2$ is denoted as $O_2 = I_1 \cup I_2$. \\

The following result mentioned in \cite{price2018} can be verified easily.

\begin{theorem}\label{price.1} 
Let $\xi\in \C_2$ be a bicomplex number. Then we have
\begin{itemize}
\item[] $\xi$ is singular $\IFF \quad \xi \in I_1 \cup I_2$
\item[]  in other words,\; $\xi$ is non-singular $\IFF \quad \xi \notin I_1 \cup I_2$.
\end{itemize}
\end{theorem}

\noindent{\bf Idempotent representation of bicomplex numbers:} \cite{price2018}
Every bicomplex number $\xi = {z}_{1}+ {i}_{2} {z}_{2}$ can  be uniquely represented as the complex combination of elements $e_{1}$ and $e_{2}$  in the following form:
\[
\xi =  (z_{1} -  i_{1} z_{2}) e_{1} + (z_{1}  +  i_{1} z_{2}) e_{2},
\]
where complex numbers $(z_{1} -  i_{1} z_{2})$ and $(z_{1} +  i_{1} z_{2})$ are called idempotent components of $\xi$ and will be denoted by $\xi^-$ and $\xi^+$ respectively. Therefore the number $\xi = {z}_{1}+ {i}_{2} {z}_{2}$ can also be written as $\xi = \xi^- e_{1} + \xi^+ e_{2}$, where $\xi^-  = z_{1} - i_{1} z_{2}$ and  $\xi^+ = z_{1}  + i_{1} z_{2}$. 
Moreover the idempotent representation of the product of elements $\xi, \eta\in\C_2$ can be seen easily to be as:
\[
\xi\cdot\eta = \left(\xi^- \eta^-\right)e_{1} + \left(\xi^+ \eta^+\right)e_{2}.
\]

\noindent{\bf Cartesian product:} \cite{anjali2024matrix}
The $n$-times Cartesian product of $\C_2$ denoted by $\C_2^n$ and consists of all $n$-tuples of bicomplex numbers of the form $(\xi_{1}, \xi_{2},\ldots,\xi_{n})$. Each component  $\xi_{i}$ belongs to $\C_2$ for $i = 1,2,\ldots,n$. Formally, this is expressed as:
\[
{\C}_{2}^{n} \eqqcolon  \{(\xi_{1}, \xi_{2},\ldots,\xi_{n}) : \;\xi_{i} \in \C_2; \;  i = 1,2,\ldots,n\}.
\]

Furthermore, using idempotent representation, each element of $\C_2^n$ can be expressed uniquely as follows:
\[
(\xi_{1}, \xi_{2},\ldots,\xi_{n}) =  (\xi_{1}^-, \xi_{2}^-,\ldots, \xi_{n}^-)e_{1}  +  (\xi_{1}^+,  \xi_{2}^+,\ldots,\xi_{n}^+)e_{2},
\]
such that $(\xi_{1}^-, \xi_{2}^-,\ldots, \xi_{n}^-)$ and $(\xi_{1}^+,  \xi_{2}^+,\ldots,\xi_{n}^+)$ are  n-tuples of complex numbers from the space $\C_1^n$.\\

 \noindent{\bf Cartesian idempotent product:} \cite{anjali2024matrix}
The Cartesian product in bicomplex spaces is defined as follows: for any subsets \;$S_1,S_2 \subseteq \C_1^{n}$, their idempotent product, denoted $S_1\times_e S_2$,  is defined as a subset of $\C_2^{n}$ given by

\begin{eqnarray}\label{idem.def}
S_1 \times_e S_2 \eqqcolon \{e_1 x + e_2 y : x \in S_1, y \in S_2\}. 
\end{eqnarray}
Analogous to the concept of cartesian product of bicomplex number,  for any subsets \;$H_1,H_2 \subseteq \C_2^{m \times n}$, their idempotent product, denoted as $H_1\times_e H_2$ is defined to be a subset of $\C_2^{m \times n}$ given as:
\begin{eqnarray}\label{idemmatrix.def}
H_1 \times_e H_2 \eqqcolon \{e_1 A + e_2 B : A \in H_1, B \in H_2\}. 
\end{eqnarray}
\justifying { Similarly for linear maps, for any non-empty subsets $A_1,A_2$ of the space of \;$\C_1$-linear maps  $\Hom(\C_1^n,\C_1^m)$, their
idempotent product, written as $A_1\times_e A_2$, is defined to be a subset of $\Hom(\C_2^n,\C_2^m)$ given as}
\begin{eqnarray}
A_1 \times_e A_2 \eqqcolon \{e_1 T_1 + e_2 T_2 : \;T_1 \in A_1, T_2 \in A_2\}. 
\end{eqnarray}

\begin{remark}
 For the convenience, we denote the set of all $\C_1$-linear maps from $\C_1^{n}$ to $\C_1^{m}$ by $L_1^{nm}$, instead of $\Hom(\C_1^{n},\C_1^{m})$ and set  of all $\C_1$-linear maps from $\C_2^{n}$ to $\C_2^{m}$ by $L_2^{nm}$.  Clearly both $L_1^{nm}$ and  $L_2^{nm}$ are vector spaces over $\C_1$. Hence, we can see that 
\begin{eqnarray}\label{4mn}
\dim(L_1^{nm}) = m n \;\;\mbox{and}\;\; \dim(L_2^{nm}) = 4mn. 
\end{eqnarray}
\end{remark}

 \begin{definition} (\textbf{Bicomplex matrix})\label{matrix def}:  
 	A bicomplex matrix of order $m\times n$ is denoted by \;$ A=[\xi_{i j}]_{m \times n}$, $\xi_{i j} \in \C_2$\ . We let the set of all bicomplex matrices of order $m \times n$ be denoted by $\C_2^{m \times n}$, i.e., we have
\begin{eqnarray}
\C_2^{m \times n}  \eqqcolon  \Big\{[\xi_{i j}] : \; \xi_{i j} \in \C_2 ;\; i = 1,2,\ldots,m,\; j = 1,2,\ldots,n \Big\}.
\end{eqnarray} 
The usual matrix addition and scalar multiplication make the space  $\C_2^{m \times n}$ to be a vector space over the field $\C_1$. This immediately indicates that the dimension of $\C_2^{m \times n}$ over $\C_1$ is 2mn, i.e. we have
\begin{eqnarray}
    \dim (\C_2^{m \times n}(\C_1)) = 2mn.
\end{eqnarray}

Analogous to the concept of bicomplex numbers, any bicomplex matrice $A=\left[\xi_{i j}\right]_{m \times n}\in \C_2^{m \times n}$ can be decomposed uniquely as 
\begin{eqnarray}
 A = e_1 \ A^- \ + \ e_2 \  A^+,
\end{eqnarray}
where  $A^- =\left[\xi^-_{i j}\right]_{m \times n},\; A^+ =\left[\xi^+_{i j}\right]_{m \times n}$ are complex matrices.

\begin{definition}\label{defid}
(see definition \cite[2.3]{anjali2024matrix}) For any given $T_{1}, T_{2} \in L_1^{nm}$,  we can define a map $T \colon \C_2^{n}  \to \C_2^{m} $  by the following rule:
\[
T(\xi_{1}, \xi_{2},\ldots,\xi_{n}) \eqqcolon e_{1}\cdot {T}_{1} (\xi_{1}^-, \xi_{2}^-,\ldots,\xi_{n}^-)
+ e_{2}\cdot {T}_{2} (\xi_{1}^+, \xi_{2}^+,\ldots,\xi_{n}^+).
\]
\end{definition}
\noindent Clearly \;$T$\; is a $\C_1$-linear map.\;$T$\; can also be represented by $e_{1} T_{1} + e_{2} T_{2}$. Thus the set of all such linear maps is the idempotent product $L_1^{nm} \times_{e} L_1^{nm}$, i.e., we have
\begin{eqnarray}\label{idemproduct}
L_1^{nm} \times_{e} L_1^{nm} &\eqqcolon & \left \{\;e_{1} T_{1} + e_{2} T_{2} \in L_2^{nm}:\; \;T_{1}, T_{2} \in L_1^{nm}\right \}.  
\end{eqnarray}

\noindent For convenience, the set of all such type of $T= e_{1} T_{1} + e_{2} T_{2} \colon \C_2^{n}  \to \C_2^{n}$ linear operators is denoted by  $L_{1}^{n} \times_{e} L_{1}^{n}$. The idempotent product $L_1^{nm} \times_{e} L_1^{nm}$ is a subspace of $L_2^{nm}$ over the field $\C_1$. This indicates directly that $L_{1}^{nm} \times_{e} L_{1}^{nm}$ has dimension $2mn$. That is 
\begin{eqnarray}
    \dim (L_1^{nm} \times_{e} L_1^{nm}(\C_1))= 2mn.
\end{eqnarray}

The vector spaces $\C_2^{m \times n}$ and $L_{1}^{nm} \times_{e} L_{1}^{nm}$ are isomorphic because they have the same dimension over the field $\C_1$. So the matrix representation of $T= e_1 T_1 + e_2 T_2$ with respect to the  ordered bases $\mathcal{B}_{1}, \mathcal B_{2}$  for $\C_1^{n}$ and $\C_1^{m}$ respectively is defined as
\begin{eqnarray}
[T]^{\mathcal{B}_{1}}_{\mathcal{B}_{2}}\eqqcolon e_1 [T_1]^{\mathcal{B}_{1}}_{\mathcal{B}_{2}} + e_{2} [T_2]^{\mathcal{B}_{1}}_{\mathcal{B}_{2}}.
\end{eqnarray}
Where $[T_1]^{\mathcal{B}_{1}}_{\mathcal{B}_{2}}$ and $[T_2]^{\mathcal{B}_{1}}_{\mathcal{B}_{2}}$ are called the matrix of $T_1$ and $T_2$ relative to the pair of ordered 
bases $\mathcal{B}_{1}$ and $\mathcal{B}_{2}$. 
In the special case when $\C_1^{n} = \C_1^{m}$, the matrix representation of the operator $T= e_{1} T_{1} + e_{2} T_{2}$ with respect to basis $\mathcal{B}$ for $\C_1^{n}$ is simplified to  $[T]_{\mathcal{B}}$ from $[T]^{\mathcal{B}}_{\mathcal{B}}$. Thus, we have
\begin{eqnarray}\label{matrixrepre}
[T]_{\mathcal{B}} = e_{1} [T_{1}]_{\mathcal{B}} + e_{2} [T_{2}]_{\mathcal{B}}.
\end{eqnarray}

\end{definition}
\begin{proposition} (See \cite{anjali2024matrix}) \label{Proposition1}
Let $T, S \in L_{1}^{nm} \times_{e} L_{1}^{nm}$ be any elements such that $T = e_{1} T_{1} + e_{2} T_{2}$ and $S = e_{1} S_{1} + e_{2} S_{2}$. Then, we have 
\begin{enumerate}
\item $T + S = e_{1} (T_{1} + S_{1}) + e_{2} (T_{2} + S_{2})$.
\item $\alpha T = e_{1} (\alpha T_{1}) + e_{2} (\alpha T_{2}); \quad \forall \alpha \in \C_1$.
\end{enumerate}
\end{proposition}

\begin{theorem}\label{Kernel} (See, \cite{anjali2024matrix})
Let $T \in L_{1}^{nm} \times_{e} L_{1}^{nm}$ so that $T=e_{1} T_{1} + e_{2} T_{2}$ for some $T_1,T_2\in L_1^{nm}$. Then, we have
\begin{enumerate}
\item $\ker(e_{1} T_{1} + e_{2}  T_{2}) = \ker T_{1} \times_{e} \ker T_{2} $.
\item $\ran (e_{1} T_{1} + e_{2}  T_{2}) = \ran T_{1} \times_{e} \ran T_{2}  $.
\end{enumerate}
\end{theorem}

\section{Eigenvalues and eigenvectors associated with $T= e_1 T_1 + e_2 T_2$}
Earlier, \cite {anjali2024matrix} had given an "Idempotent method" for matrix representation of a linear map of the type $T= e_1 T_1 + e_2 T_2 : \C_2^{n} \longrightarrow \C_2^{m} $. This method gives a one-one correspondence between the bicomplex matrices $A=[\xi_{i j}]_{n \times n}$ and the linear operator's $T= e_1 T_1 + e_2 T_2 $ on finite dimensional vector space $\C_2^{n}$. This unique correspondence can help to develop the concept of eigenvalue and eigenvector of a bicomplex matrix $A=[\xi_{i j}]_{n \times n}$. For this reason, this linear operator $T$ is vital for our analysis and plays an essential role from the bicomplex matrix point of view. We refer the reader to \cite {anjali2024matrix} for further details. This section aims to analyze the facts and find the essential results related to eigenvalue and eigenvector of $T= e_1 T_1 + e_2 T_2 $.

\begin{definition}(Non-singular and Singular vectors in $\C_2^n$) \label{rem 1}:
A vector $(\xi_{1},\xi_{2},\ldots,\xi_{n}) \in \C_2^n$ is non-singular if at least one of $\xi_i, i=1,2 \ldots,n$ is non-singular otherwise it is singular. Furthermore, for any $(\xi_{1},\xi_{2},\ldots,\xi_{n}) \in \C_2^n$ is non-singular then $(\xi_{1}^-, \xi_{2}^-,\ldots,\xi_{n}^-) \neq 0$ and $(\xi_{1}^+, \xi_{2}^+,\ldots,\xi_{n}^+)  \neq 0$ but converse need not be true.  
\end{definition}

\begin{example}
	Here, $(1,0,\ldots,0)$ and $(0,1,\ldots,0)$ are non-zero vectors of $\C_1^n$ but $ e_1 (1,0,\ldots,0) +  e_2 (0,1,\ldots,0)= (e_1,e_2,\ldots,0)$ is a singular vector of $\C_2^n$.
\end{example}

\begin{remark}\label{rem 1.1}
If $0 \neq (\xi_{1},\xi_{2},\ldots,\xi_{n})   \in \C_2^n $  \quad \IFF \quad either $ (\xi_{1}^-, \xi_{2}^-,\ldots,\xi_{n}^-)  \neq 0 $ \ or \ $ (\xi_{1}^+, \xi_{2}^+,\ldots,\xi_{n}^+)  \neq 0$.   
\end{remark}
\begin{definition} \label{def1 Ev}(Modified eigenvalue and Modified eigenvector)
: Let $T= e_1 T_1 + e_2 T_2 \in L_1^{n} \times_e  L_1^{n} $ be a linear operator on $\C_2^n(\C_1)$. A scalar,  $\lambda \in \C_1$ is an \textbf{eigenvalue} of $T$ if there exist a non-zero vector  $ (\xi_{1},\xi_{2},\ldots,\xi_{n}) \in \C_2^n$ such that $T (\xi_{1},\xi_{2},\ldots,\xi_{n}) = \lambda 
(\xi_{1},\xi_{2},\ldots,\xi_{n})$. Then the vector $(\xi_{1},\xi_{2},\ldots,\xi_{n})$ is called an \textbf{eigenvector} of $T$ associated with the eigenvalue $\lambda \in \C_1$. Furthermore, if $\kappa \in \C_2$ with $T (\xi_{1},\xi_{2},\ldots,\xi_{n}) = \kappa (\xi_{1},\xi_{2},\ldots,\xi_{n})$, then $\kappa$ is called a \textbf{modified eigenvalue} of  $T$ and the non zero vector $ (\xi_{1},\xi_{2},\ldots,\xi_{n})$ is called a \textbf{modified eigenvector} of $T$ associated with the modified eigenvalue $\kappa$.
\end{definition}

\begin{remark}
	Every eigenvalue (eigenvector) of $T= e_1 T_1 + e_2 T_2$ is a modified eigenvalue (modified eigenvector) of T but converse need not be true.
\end{remark}

\begin{example}\label{26}
Let $T_1, T_2  \colon \C_1^{2} \to \C_1^{2}$ be two linear operators such that 
\[
T_1(x,y)=(x,0)\quad\mbox{and}\quad T_2(x,y) = (x,y)\quad\forall\;(x,y)\in \C_1^2.\]

It is evident that  $T=e_1 T_1 + e_2 T_2 \colon \C_2^2(\C_1)  \to \C_2^2(\C_1)$ such that $T(\xi, \eta)= (\xi, \eta e_2)$. It is a routine matter to check that $(1, e_2)$ is an eigenvector (modified eigenvector) of T associated with the eigenvalue (modified eigenvalue) 1. Moreover, $(e_1, 0)$ is a modified eigenvector of T associated with the modified eigenvalue $1e_1 + re_2$, for $r\in \C_1$.  
\end{example}

\begin{theorem} \label{th1}
A linear operator $T = e_1T_1 + e_2T_2 \in L_1^{n} \times_e L_1^{n}$ is singular $\IFF$  either $T_1$ is singular or $T_2$ is singular.
\end{theorem}

\begin{proof}
Suppose that $T=e_1 T_1 + e_2 T_2 \colon \C_2^n(\C_1)  \to \C_2^n(\C_1)$ is singular. Using \ref{Kernel} and \ref{rem 1.1}, it follows that 
\begin{eqnarray*}
  ker(e_1 T_1 + e_2 T_2) \neq \{0\}
   &\Leftrightarrow&  ker(T_1 ) \times_e ker(T_2) \neq \{0\}\\ 
  &\Leftrightarrow& \mbox{there exists a non-zero vecctor}\\ 
  && (z_1,z_2,\ldots,z_n) e_1 + (w_1,w_2,\ldots,w_n) e_2 \in ker(T_1 ) \times_e ker(T_2)\\
  && \mbox{such that either}\ 0\neq (z_1,z_2,\ldots,z_n) \in ker(T_1 )\ \mbox{or}\\
  && 0\neq (w_1,w_2,\ldots,w_n) \in ker(T_2 )\\
  &\Leftrightarrow& \mbox {either} \quad ker(T_1)  \neq \{0\} \quad \mbox{or} \quad  ker(T_2)  \neq \{0\} \\
   &\Leftrightarrow&  \mbox {either} \quad  T_1 \quad \mbox {is singular} \quad \mbox{or} \quad T_2 \quad   \mbox {is singular}.
\end{eqnarray*} 
 Thus the theorem is proved.
\end{proof}
\begin{definition} \label{def 2}
	If $\eta= \eta^- e_1 + \eta^+ e_2 \in \C_2$ and $T= e_1 T_1 + e_2T_2 \in L_1^{nm} \times_{e} L_1^{nm}$, then we define a linear transformation $\eta T$ as
	\begin{center}
		$\eta \ T = e_1 (\eta^- \ T_1) + e_2 (\eta^+ \
		T_2)$.
	\end{center}
	 It is a trivial matter to check that $\eta \ T \in L_1^{nm} \times_{e} L_1^{nm}$
	  and if 
	  $I$ represents an identity linear operator then $I=e_1 I + e_2I  \in L_1^{n} \times_e L_1^{n}$ and 
	\begin{center}
		$\eta \ I = e_1 (\eta^- \ I) + e_2 (\eta^+ \
		I)$.
	\end{center}
\end{definition}
\begin{theorem}\label{210}
	Let $T= e_1 T_1 + e_2T_2 \in L_1^{n} \times_{e} L_1^{n}$, $I=e_1 I + e_2I  \in L_1^{n} \times_e L_1^{n}$ be the identity linear operator and $\kappa= \kappa_1 e_1 + \kappa_2 e_2 \in \C_{2}$, then $ (T-\kappa I)$ is singular if and only if either $(T_1 - \kappa_1 I)$ is singular or $(T_2 - \kappa_2 I)$ is singular. 
\end{theorem}
\begin{proof}
	Let us suppose that $\kappa= \kappa_1 e_1 + \kappa_2 e_2 \in \C_{2}$ and $ (T-\kappa I)$ be singular. We use \ref{Proposition1} and \ref{th1} throughout this proof.\\
	 $ \Leftrightarrow (e_1 T_1 + e_2 T_2)+\{e_1(-\kappa_1I)+ e_2(-\kappa_2I)\}$ is singular.\\
	$ \Leftrightarrow e_1 (T_1-\kappa_1I) + e_2 (T_2-\kappa_2I)$ is singular.\\ 
	$\Leftrightarrow$ either $(T_1-\kappa_1I)$ is singular or $(T_2-\kappa_2I)$ is singular.\\
	Thus the theorem is proved.
\end{proof}
Dually we may prove the following corollary for $ \lambda \in \C_1$.
\begin{corollary}\label{31}
	Let $T= e_1 T_1 + e_2T_2 \in L_1^{n} \times_{e} L_1^{n}$, $I=e_1 I + e_2I  \in L_1^{n} \times_e L_1^{n}$ be the identity linear operator and $\lambda \in \C_{1}$, then $ (T-\lambda I)$ is singular if and only if either $(T_1 - \lambda I)$ is singular or $(T_2 - \lambda I)$ is singular. 
\end{corollary}

\section{Relationship between the set of eigenvalues and the set of modified eigenvalues} 
In this section, we intend to define and analyze three sets of significant importance for the eigenvalue of $T= e_1 T_1 + e_2  T_2$. These sets are crucial in understanding the eigenvalue and modified eigenvalue of $T=e_1 T_1 + e_2 T_2 $. We also deal eigenspace and modified eigenspace and discuss some related results.
\begin{theorem}\label{29}
	Let $T= e_1 T_1 + e_2T_2 \in L_1^{n} \times_{e} L_1^{n}$ and $I=e_1 I + e_2I  \in L_1^{n} \times_e L_1^{n}$ be the identity linear operator, then a bicomplex number $\kappa= \kappa_1 e_1 + \kappa_2 e_2$ is a modified eigenvalue of $T$ if and only if $ (T-\kappa I)$ is singular.
\end{theorem}
\begin{proof}
	Let $\kappa \in \C_2$ be a modified eigenvalue of $T $. We use \ref{rem 1.1} and \ref{210} throughout this proof. This implies that there exists a non zero vector  $ (\xi_1,\xi_2,\ldots,\xi_n) $ such that
	\begin{eqnarray*}
		T (\xi_{1},\xi_{2},\ldots,\xi_{n}) &=& \kappa (\xi_{1},\xi_{2},\ldots,\xi_{n}),\ (\xi_{1},\xi_{2},\ldots,\xi_{n}) \neq 0\\
		\Leftrightarrow  e_{1} {T}_{1} (\xi_{1}^-, \xi_{2}^-,\ldots,\xi_{n}^-)
		+ e_{2} {T}_{2} (\xi_{1}^+, \xi_{2}^+,\ldots,\xi_{n}^+) &=&  e_1 \kappa_1 (\xi_{1}^-, \xi_{2}^-,\ldots,\xi_{n}^-) + e_2 \kappa_2 (\xi_{1}^+, \xi_{2}^+,\ldots,\xi_{n}^+)\\
		\Leftrightarrow  {T}_{1} (\xi_{1}^-, \xi_{2}^-,\ldots,\xi_{n}^-)= \kappa_1(\xi_{1}^-, \xi_{2}^-,\ldots, \xi_{n}^-) &\mbox{or}& {T}_{2} (\xi_{1}^+, \xi_{2}^+,\ldots,\xi_{n}^+)= \kappa_2(\xi_{1}^+, \xi_{2}^+,\ldots, \xi_{n}^+)\\
		\mbox{such that}\ (\xi_{1}^-, \xi_{2}^-,\ldots, \xi_{n}^-) \neq 0 && \mbox{such that}\ (\xi_{1}^+, \xi_{2}^+,\ldots,\xi_{n}^+) \neq 0 
	\end{eqnarray*}
	either $\kappa_1$ is eigenvalue of $T_1$ or $\kappa_2$ is eigenvalue of $T_2$\\
	$\Leftrightarrow$ either $(T_1-\kappa_1I)$ is singular or $(T_2-\kappa_2I)$ is singular\\
	$ \Leftrightarrow (T - \kappa I)$ is singular.\\
	Thus the theorem is proved.
\end{proof}
Dually we may prove the following corollary for $ \lambda \in \C_1$.
\begin{corollary}\label{rem2.1} 
	Let $T= e_1 T_1 + e_2T_2 \in L_1^{n} \times_{e} L_1^{n}$ and $I=e_1 I + e_2I  \in L_1^{n} \times_e L_1^{n}$ be the identity linear operator then, a complex number $\lambda$ is an eigenvalue of $T$ if and only if $ (T-\lambda I)$ is singular. 
\end{corollary}

\begin{definition}
Let $T= e_1 T_1 + e_2T_2 \in L_1^{n} \times_{e} L_1^{n}$ then the set of all modified eigenvalues of $T$, eigenvalues of $T_1$ and  $T_2$
are denoted by the sets  $\Upsilon, \Upsilon_1$ and $\Upsilon_2$ respectively, that is
\begin{eqnarray*}
   \Upsilon &\eqqcolon& \mbox{Set of all modified eigenvalues of} \ T\\
 \Upsilon_1 &\eqqcolon& \mbox{Set of all eigenvalues of} \ T_1\\
 \Upsilon_2 &\eqqcolon& \mbox{Set of all  eigenvalues of} \ T_2  .
\end{eqnarray*}
 \end{definition}
The idempotent product of the sets  $\Upsilon_1 $   and $\Upsilon_2$ is 
\begin{center}
	$\Upsilon_1 \times_e \Upsilon_2 \eqqcolon \{\kappa_1 e_1 + \kappa_2 e_2: \kappa_1 \in \Upsilon_1 , \kappa_2 \in \Upsilon_2\}$ .   
\end{center}
\noindent{By defining and examining these sets, we can get valuable results for an eigenvalue and modified eigenvalue of  $T=e_1 T_1 + e_2 T_2$.}

\noindent{The next theorem connects the set $\Upsilon$ with the set   $\Upsilon_1 \times_e \Upsilon_2 $.}

\begin{theorem} \label{th5}
The set $\Upsilon_1 \times_e \Upsilon_2 $ is contained in the set $\Upsilon$, that is 
\begin{center}
 $\Upsilon_1 \times_e \Upsilon_2 \subset \Upsilon$.   
\end{center}
\end{theorem}
\begin{proof}
Let us suppose that $\kappa = \kappa_1 e_1 + \kappa_2 e_2 \in \Upsilon_1 \times_e \Upsilon_2 $, then
\begin{eqnarray*}
  && \kappa_1 \ \mbox{is an eigenvalue of } \  T_1 \ \mbox{and}  \  \kappa_2 \  \mbox{is an eigenvalue of } \ T_2 \\
  &\Rightarrow& \exists \quad 0 \neq  (v_1, v_2,\ldots ,v_n) \in \C_1^n  \quad \mbox{and} \quad 0\neq  (w_1, w_2,\ldots, w_n) \in \C_1^n \quad  \mbox{such that} \quad\\
& & T_1(v_1, v_2,\ldots, v_n)= \kappa_1 (v_1, v_2,\ldots ,v_n) \quad \mbox{and} \quad T_2(w_1, w_2,\ldots, w_n)= \kappa_2 (w_1, w_2,\ldots ,w_n)\\
&\Rightarrow& e_1 T_1(v_1, v_2,\ldots, v_n) + e_2 T_2(w_1, w_2,\ldots ,w_n) = e_1 \kappa_1 (v_1, v_2,\ldots, v_n)  + e_2 \kappa_2 (w_1, w_2,\ldots ,w_n) \\
&\Rightarrow& (e_1 T_1 + e_2 T_2) [e_1 (v_1, v_2,\ldots ,v_n) + e_2(w_1, w_2,\ldots ,w_n)]=\kappa [e_1(v_1, v_2,\ldots ,v_n)  +   e_2(w_1, w_2,\ldots, w_n) ]\\
&& \mbox{In view of} \ \ref{rem 1.1},\ e_1(v_1, v_2,\ldots, v_n)  +   e_2(w_1, w_2,\ldots, w_n) \neq 0 \\
&\Rightarrow& \kappa= \kappa_1 e_1 + \kappa_2 e_2 \quad  \mbox{is a modified eigenvalue of} \quad T=e_1 T_1 + e_2 T_2.\\
&\Rightarrow& \kappa = \kappa_1 e_1 + \kappa_2 e_2 \in \Upsilon. 
\end{eqnarray*}
Therefore, $\Upsilon_1 \times_e \Upsilon_2 \subset \Upsilon$.\\
Thus the theorem is proved.
\end{proof}
\begin{remark}
	In view of \ref{26}, the set $\Upsilon_1=\left\lbrace 0, 1\right\rbrace $ and $\Upsilon_2=\left\lbrace 1\right\rbrace $ but $1e_1 + 2e_2$ is a modified eigenvalue of $T=e_1T_1+e_2T_2$. It yields that the set $\Upsilon_1 \times_e \Upsilon_2 $ is properly contained in the set $\Upsilon$. That is, in case of \ref{26},
\begin{eqnarray}\label{11}
\Upsilon_1 \times_e \Upsilon_2 \subsetneq \Upsilon.
\end{eqnarray}
\end{remark}
We will prove further (\ref{11}) for arbitrary $\Upsilon$,  $\Upsilon_1$ and $\Upsilon_2$. The following corollary \ref{36} are immediate consequences of \ref{31} and \ref{rem2.1}. 
\begin{corollary}\label{36}
Let $T=e_1 T_1 + e_2 T_2 \in L_1^{n} \times_{e} L_1^{n}$ , then  $\lambda \in \C_1$ is an eigenvalue of $T$ $\IFF$ $\lambda \in \Upsilon_1 \cup \Upsilon_2$.
\end{corollary}
The following corollary 3.7 are immediate consequences of \ref{210} and \ref{29}.
\begin{corollary}\label{37}
	Let $T=e_1 T_1 + e_2 T_2 \in L_1^{n} \times_{e} L_1^{n}$, then  $\kappa =\kappa_1e_1 +\kappa_2 e_2 \in \C_2$ is a modified eigenvalue of $T$ $\IFF$ either $\kappa_1 \in \Upsilon_1 $ or $\kappa_2 \in \Upsilon_2 $.
  \end{corollary}
  \begin{theorem} \label{th8}
  	Let $T=e_1 T_1 + e_2 T_2 \in L_1^{n} \times_{e} L_1^{n}$, then a modified eigenvalue of $T=T_1e_1 + T_2 e_2$ exists $\IFF$ an eigenvalue of $T$ exists.
  \end{theorem}
  \begin{proof}
  	Suppose that a modified eigenvalue of $T$ exists. We use \ref{36} and \ref{37}.
  	\begin{eqnarray*}
  		&\Rightarrow& \mbox{there exists} \ \kappa =\kappa_1 e_1 +\kappa_2 e_2 \in \Upsilon \ \mbox{such that} \ \mbox{either} \quad \kappa_1 \in \Upsilon_1  \quad \mbox{or} \quad \kappa_2 \in \Upsilon_2.\\
  		&\Rightarrow& \kappa_1 \in \Upsilon_1 \cup \Upsilon_2 \ \mbox{or} \ \kappa_2 \in \Upsilon_1 \cup \Upsilon_2 \\
  	&\Rightarrow& \mbox{either} \quad \kappa_1  \quad \mbox{is an eigenvalue of} \quad T
  		\quad \mbox{or} \quad \kappa_2  \quad \mbox{is an eigenvalue of} \quad T\\
  		&\Rightarrow&\mbox{Eigenvalue of} \ T \ \mbox{exists}.
  		\end{eqnarray*}
  		The proof of the converse part is very trivial as every eigenvalue of $T$ is a modified eigenvalue of $T$.\\
 Thus the theorem is proved.
  \end{proof}
  The following theorem shows the existence of an infinite number of modified eigenvalues of $T=e_1 T_1 +e_2 T_2$.
\begin{theorem} \label{th3.1}
	Let $T=e_1 T_1 + e_2 T_2 \in L_1^{n} \times_{e} L_1^{n}$ and $\kappa_1 \in \C_1$ be an eigenvalue of $T_1$, then $ \kappa_1e_1 + w e_2$ is a modified eigenvalue of $T$, for all $w \in \C_1$.
\end{theorem}
\begin{proof}
    Let us suppose that $\kappa_1\in \C_1$ be an eigenvalue of $T_1$. This implies that there exists a non-zero vector $(v_1, v_2,\ldots ,v_n) \in \C_1^n$ such that $T_1(v_1, v_2,\ldots, v_n) = \kappa_1(v_1, v_2,\ldots ,v_n)$. By \ref{rem 1.1}, $0 \neq e_1(v_1, v_2,\ldots, v_n)  +   e_2(0, 0,\ldots, 0) \in \C_2^n$.\\
    Now,
    \begin{eqnarray*}
   T\left[ e_1(v_1, v_2,\ldots, v_n)  +   e_2(0, 0,\ldots, 0)\right] &=& (e_1T_1 +T_2e_2)\left[ e_1(v_1, v_2,\ldots, v_n)  +   e_2(0, 0,\ldots, 0)\right]\\
    &=& e_1T_1(v_1, v_2,\ldots, v_n) +e_2 T_2(0, 0,\ldots, 0)\\
    &=& e_1\left[ \kappa_1(v_1, v_2,\ldots ,v_n)\right] +e_2 \left[w(0, 0,\ldots, 0)\right]; \quad \forall w\in \C_1\\
    &=& (\kappa_1e_1 + w e_2)\left[ e_1(v_1, v_2,\ldots, v_n)  +   e_2(0, 0,\ldots, 0)\right].
    \end{eqnarray*}
     Therefore $(\kappa_1e_1 + w e_2)$ is modified eigenvalue of $T$.\\
  Thus the theorem is proved.   
\end{proof}
The following result exhibits analogous to the concept of the theorem \ref{th3.1}. We can deduce the same result for the $T_2$ linear operator.
\begin{corollary} \label{th4} 
		Let $T=e_1 T_1 + e_2 T_2 \in L_1^{n} \times_{e} L_1^{n}$ and $\kappa_2 \in \C_1$ be an eigenvalue of $T_2$, then $ ze_1 + \kappa_2 e_2$ is a modified eigenvalue of $T$, for all $z \in \C_1$.
  \end{corollary}
 \begin{theorem}\label{th7} 
The set $\Upsilon$ is equal to the set $\Upsilon = (\Upsilon_1 \times_e \C_1) \cup (\C_1 \times_e \Upsilon_2)$, that is
     \begin{center}
     	 $\Upsilon = (\Upsilon_1 \times_e \C_1) \cup (\C_1 \times_e \Upsilon_2)$.
     \end{center}
    \end{theorem}

\begin{proof}
    Let us suppose that $\kappa= \kappa_1 e_1 + \kappa_2 e_2 \in \Upsilon$ . 
    \begin{eqnarray*}
    	&\Rightarrow& \kappa \ \mbox{is a modified eigenvalue of} \ T=e_1T_1 +e_2 T_2.\\
    	&& \mbox{By \ref{37}, we have}\\
        &\Leftrightarrow& \mbox{either} \ \kappa_1 \in \Upsilon_1 \ \mbox{or} \ \kappa_2 \in \Upsilon_2 \\
         &\Leftrightarrow&   \mbox{either} \ \kappa= \kappa_1 e_1 + \kappa_2 e_2 \in (\Upsilon_1 \times_e \C_1) \ \mbox{or} \ \kappa= \kappa_1 e_1 + \kappa_2 \in (\C_1 \times_e \Upsilon_2) \\
         &\Leftrightarrow&    \kappa= \kappa_1 e_1 + \kappa_2 e_2 \in (\Upsilon_1 \times_e \C_1) \cup (\C_1 \times_e \Upsilon_2).
    \end{eqnarray*}
   Thus the theorem is proved.    
\end{proof}
\begin{remark}
It is evident that $\Upsilon_1 \times_e \Upsilon_2 \subsetneq (\Upsilon_1 \times_e \C_1) \cup (\C_1 \times_e \Upsilon_2)$, that is for arbitrary $\Upsilon$,  $\Upsilon_1$ and $\Upsilon_2$, $\Upsilon_1 \times_e \Upsilon_2 \subsetneq \Upsilon.$
\end{remark}
\begin{definition}(Eigenspace and Modified eigenspace)\label{313} :
	Let $T= e_1 T_1 + e_2 T_2 \in L_1^{n} \times_e  L_1^{n} $ be a linear operator on $\C_2^n(\C_1)$, and $\lambda$ and $\kappa= \kappa_1 e_1 + \kappa_2 e_2$ be the eigenvalue and modified eigenvalue of $T$.
	The eigenspace and modified eigenspace of $T$ associated with the eigenvalue $\lambda$ and modified eigenvalue $\kappa$ are denoted by $^T(E)_\lambda$  and $^T(ME)_{\kappa}$, respectively and defined as
	\begin{eqnarray*}
	^T(E)_\lambda &=& \left\lbrace (\xi_{1},\xi_{2},\ldots,\xi_{n}) \in \C_2^n:T (\xi_{1},\xi_{2},\ldots,\xi_{n}) = \lambda 
	(\xi_{1},\xi_{2},\ldots,\xi_{n})\right\rbrace.\\
	^T(ME)_{\kappa} &=& \left\lbrace (\eta_{1},\eta_{2},\ldots,\eta_{n}) \in \C_2^n:T (\eta_{1},\eta_{2},\ldots,\eta_{n}) = \kappa 
	(\eta_{1},\eta_{2},\ldots,\eta_{n})\right\rbrace.
\end{eqnarray*}
It is a trivial matter to check that the modified eigenspace $^T(ME)_{\kappa}$ of $T$ associated with the modified eigenvalue $\kappa$ forms a subspace of $\C_2^n$. The eigenspace of $T_1$ and $T_2$ associated with the eigenvalue $z$  and $w$ are denoted by $^{T_1}(E)_z$ and $^{T_2}(E)_w$, respectively.
\end{definition}
The following theorems \ref{314}, \ref{316} and \ref{318} are immediate consequence of \ref{37} and \ref{313}.
\begin{theorem}\label{314}
	Let $T= e_1 T_1 + e_2 T_2 \in L_1^{n} \times_e  L_1^{n} $ be a linear operator on $\C_2^n(\C_1)$ and $\kappa= \kappa_1 e_1 + \kappa_2 e_2$ be the modified eigenvalue of $T$ such that $\kappa_1 \in \Upsilon_1 $ and $\kappa_2 \notin \Upsilon_2 $, then
	\begin{center}
		$^T(ME)_{\kappa}={}^{T_1}(E)_{\kappa_1} \times_e \left\lbrace 0\right\rbrace $
	\end{center}
\end{theorem} 
\begin{proof} 
Let us Suppose that $(\eta_{1},\eta_{2},\ldots,\eta_{n}) \in {}^T(ME)_{\kappa}$.
\begin{eqnarray*}
	\Leftrightarrow T (\eta_{1},\eta_{2},\ldots,\eta_{n}) &=& \kappa 
	(\eta_{1},\eta_{2},\ldots,\eta_{n})\\
	\Leftrightarrow (e_1T_1 + e_2 T_2)(\eta_{1},\eta_{2},\ldots,\eta_{n}) &=& (\kappa_1 e_1 + \kappa_2 e_2) 
	(\eta_{1},\eta_{2},\ldots,\eta_{n})\\
	\Leftrightarrow e_{1} {T}_{1} (\eta_{1}^-, \eta_{2}^-,\ldots,\eta_{n}^-)
	+ e_{2} {T}_{2} (\eta_{1}^+, \eta_{2}^+,\ldots,\eta_{n}^+) &=&  e_1 \kappa_1 (\eta_{1}^-, \eta_{2}^-,\ldots,\eta_{n}^-) + e_2 \kappa_2 (\eta_{1}^+, \eta_{2}^+,\ldots,\eta_{n}^+)\\
	\Leftrightarrow {T}_{1} (\eta_{1}^-, \eta_{2}^-,\ldots,\eta_{n}^-)
	= \kappa_1 (\eta_{1}^-, \eta_{2}^-,\ldots,\eta_{n}^-) &\mbox{and}& {T}_{2} (\eta_{1}^+, \eta_{2}^+,\ldots,\eta_{n}^+)=\kappa_2 (\eta_{1}^+, \eta_{2}^+,\ldots,\eta_{n}^+)\\
	\Leftrightarrow (\eta_{1}^-, \eta_{2}^-,\ldots,\eta_{n}^-) \in {}^{T_1}(E)_{\kappa_1} &\mbox{and}& (\eta_{1}^+, \eta_{2}^+,\ldots,\eta_{n}^+) = 0, \ \mbox{as} \ \kappa_2 \notin \Upsilon_2  \\
	\Leftrightarrow (\eta_{1},\eta_{2},\ldots,\eta_{n}) &\in& {}^{T_1}(E)_{\kappa_1} \times_e \left\lbrace 0\right\rbrace
	\end{eqnarray*}
	It shows that $^T(ME)_{\kappa}={}^{T_1}(E)_{\kappa_1} \times_e \left\lbrace 0\right\rbrace $.\\
Thus the theorem is proved.
\end{proof}
The following corollary \ref{315} is immediate consequence of \ref{314}.
\begin{corollary}\label{315}
	Let $T= e_1 T_1 + e_2 T_2 \in L_1^{n} \times_e  L_1^{n} $ be a linear operator on $\C_2^n(\C_1)$ and $\kappa= \kappa_1 e_1 + \kappa_2 e_2$ be the modified eigenvalue of $T$ such that $\kappa_1 \in \Upsilon_1 $ and $\kappa_2 \notin \Upsilon_2 $, then all modified eigenvector of $T$ associated with the modified eigenvalue $\kappa$ are singular, that is it will be a multiple of $e_1$ and
		\begin{center}
		$^T(ME)_{\kappa}=e_1{}^{T_1}(E)_{\kappa_1}$.
	\end{center}
\end{corollary}
\begin{theorem}\label{316}
	Let $T= e_1 T_1 + e_2 T_2 \in L_1^{n} \times_e  L_1^{n} $ be a linear operator on $\C_2^n(\C_1)$ and $\kappa= \kappa_1 e_1 + \kappa_2 e_2$ be the modified eigenvalue of $T$ such that $\kappa_1 \notin \Upsilon_1 $ and $\kappa_2 \in \Upsilon_2 $, then
	\begin{center}
		$^T(ME)_{\kappa}=\left\lbrace 0\right\rbrace \times_e {}^{T_2}(E)_{\kappa_2} $
	\end{center}
\end{theorem} 
\begin{proof} 
	Let us Suppose that $(\eta_{1},\eta_{2},\ldots,\eta_{n}) \in {}^T(ME)_{\kappa}$.
	\begin{eqnarray*}
		\Leftrightarrow T (\eta_{1},\eta_{2},\ldots,\eta_{n}) &=& \kappa 
		(\eta_{1},\eta_{2},\ldots,\eta_{n})\\
		\Leftrightarrow (e_1T_1 + e_2 T_2)(\eta_{1},\eta_{2},\ldots,\eta_{n}) &=& (\kappa_1 e_1 + \kappa_2 e_2) 
		(\eta_{1},\eta_{2},\ldots,\eta_{n})\\
		\Leftrightarrow e_{1} {T}_{1} (\eta_{1}^-, \eta_{2}^-,\ldots,\eta_{n}^-)
		+ e_{2} {T}_{2} (\eta_{1}^+, \eta_{2}^+,\ldots,\eta_{n}^+) &=&  e_1 \kappa_1 (\eta_{1}^-, \eta_{2}^-,\ldots,\eta_{n}^-) + e_2 \kappa_2 (\eta_{1}^+, \eta_{2}^+,\ldots,\eta_{n}^+)\\
		\Leftrightarrow {T}_{1} (\eta_{1}^-, \eta_{2}^-,\ldots,\eta_{n}^-)
		= \kappa_1 (\eta_{1}^-, \eta_{2}^-,\ldots,\eta_{n}^-) &\mbox{and}& {T}_{2} (\eta_{1}^+, \eta_{2}^+,\ldots,\eta_{n}^+)=\kappa_2 (\eta_{1}^+, \eta_{2}^+,\ldots,\eta_{n}^+)\\
		\Leftrightarrow (\eta_{1}^-, \eta_{2}^-,\ldots,\eta_{n}^-) = 0, \ \mbox{as} \ \kappa_1 \notin \Upsilon_1 &\mbox{and}& (\eta_{1}^+, \eta_{2}^+,\ldots,\eta_{n}^+) \in {}^{T_2}(E)_{\kappa_2}   \\
		\Leftrightarrow (\eta_{1},\eta_{2},\ldots,\eta_{n}) &\in& \left\lbrace 0\right\rbrace \times_e {}^{T_2}(E)_{\kappa_2}
	\end{eqnarray*}
	It shows that $^T(ME)_{\kappa}=\left\lbrace 0\right\rbrace \times_e {}^{T_2}(E)_{\kappa_2} $.\\
	Thus the theorem is proved.
\end{proof}
The following corollary \ref{317} is immediate consequence of \ref{316}.
\begin{corollary}\label{317}
	Let $T= e_1 T_1 + e_2 T_2 \in L_1^{n} \times_e  L_1^{n} $ be a linear operator on $\C_2^n(\C_1)$ and $\kappa= \kappa_1 e_1 + \kappa_2 e_2$ be the modified eigenvalue of $T$ such that $\kappa_1 \notin \Upsilon_1 $ and $\kappa_2 \in \Upsilon_2 $, then all modified eigenvector of $T$ associated with the modified eigenvalue $\kappa$ are singular, that is it will be a multiple of $e_2$ and
		\begin{center}
		$^T(ME)_{\kappa}=e_2{}^{T_2}(E)_{\kappa_2}$.
	\end{center}
\end{corollary}
\begin{theorem}\label{318}
	Let $T= e_1 T_1 + e_2 T_2 \in L_1^{n} \times_e  L_1^{n} $ be a linear operator on $\C_2^n(\C_1)$ and $\kappa= \kappa_1 e_1 + \kappa_2 e_2$ be the modified eigenvalue of $T$ such that $\kappa_1 \in \Upsilon_1 $ and $\kappa_2 \in \Upsilon_2 $, then
	\begin{center}
		$^T(ME)_{\kappa}={}^{T_1}(E)_{\kappa_1} \times_e {}^{T_2}(E)_{\kappa_2} $
	\end{center}
\end{theorem} 
\begin{proof} 
	Let us Suppose that $(\eta_{1},\eta_{2},\ldots,\eta_{n}) \in {}^T(ME)_{\kappa}$.
	\begin{eqnarray*}
		\Leftrightarrow T (\eta_{1},\eta_{2},\ldots,\eta_{n}) &=& \kappa 
		(\eta_{1},\eta_{2},\ldots,\eta_{n})\\
		\Leftrightarrow (e_1T_1 + e_2 T_2)(\eta_{1},\eta_{2},\ldots,\eta_{n}) &=& (\kappa_1 e_1 + \kappa_2 e_2) 
		(\eta_{1},\eta_{2},\ldots,\eta_{n})\\
		\Leftrightarrow e_{1} {T}_{1} (\eta_{1}^-, \eta_{2}^-,\ldots,\eta_{n}^-)
		+ e_{2} {T}_{2} (\eta_{1}^+, \eta_{2}^+,\ldots,\eta_{n}^+) &=&  e_1 \kappa_1 (\eta_{1}^-, \eta_{2}^-,\ldots,\eta_{n}^-) + e_2 \kappa_2 (\eta_{1}^+, \eta_{2}^+,\ldots,\eta_{n}^+)\\
		\Leftrightarrow {T}_{1} (\eta_{1}^-, \eta_{2}^-,\ldots,\eta_{n}^-)
		= \kappa_1 (\eta_{1}^-, \eta_{2}^-,\ldots,\eta_{n}^-) &\mbox{and}& {T}_{2} (\eta_{1}^+, \eta_{2}^+,\ldots,\eta_{n}^+)=\kappa_2 (\eta_{1}^+, \eta_{2}^+,\ldots,\eta_{n}^+)\\
		\Leftrightarrow (\eta_{1}^-, \eta_{2}^-,\ldots,\eta_{n}^-) \in {}^{T_1}(E)_{\kappa_1} &\mbox{and}& (\eta_{1}^+, \eta_{2}^+,\ldots,\eta_{n}^+) \in {}^{T_2}(E)_{\kappa_2}   \\
		\Leftrightarrow (\eta_{1},\eta_{2},\ldots,\eta_{n}) &\in& {}^{T_1}(E)_{\kappa_1} \times_e {}^{T_2}(E)_{\kappa_2}
	\end{eqnarray*}
	It shows that $^T(ME)_{\kappa}={}^{T_1}(E)_{\kappa_1} \times_e {}^{T_2}(E)_{\kappa_2} $.\\
	Thus the theorem is proved.
\end{proof}

\section{Open problems}
\noindent{In the view of section 3, we propose an important open problem.}\\
{\textbf{Problem 4.1}} If $T= e_1 T_1 + e_2 T_2 \in L_1^{n} \times_e  L_1^{n} $ is  a linear operator on $\C_2^n(\C_1)$, ${}^T(ME)_{\kappa}$ and $ {}^T(ME)_{\kappa'}$ are the modified eigenspaces of $T$ associated with modified eigenvalues $\kappa= \kappa_1 e_1 + \kappa_2 e_2$ and $\kappa'= \kappa'_1 e_1 + \kappa'_2 e_2$, respectively. 
Is the sum of these two modified eigenspaces, ${}^T(ME)_{\kappa}$ and $ {}^T(ME)_{\kappa'}$, equal to the direct sum of ${}^{T}(ME)_{\kappa}$ and ${}^{T}(ME)_{\kappa'}$; that is 
\begin{eqnarray*}
    {}^T(ME)_{\kappa} + {}^T(ME)_{\kappa'} = {}^{T}(ME)_{\kappa} \oplus \ {}^{T}(ME)_{\kappa}
\end{eqnarray*}

\bibliographystyle{plain}
\bibliography{references}
\end{document}